\documentclass[preprint,12pt]{elsarticle}
\usepackage{amsthm,amsfonts,amssymb,amscd,amsmath,enumerate,verbatim,calc,graphicx,geometry}
\usepackage[all]{xy}

\newtheorem{theorem}{Theorem}[section]
\newtheorem{lemma}[theorem]{Lemma}

\newtheorem{corollary}[theorem]{Corollary}
\theoremstyle{definition}
\theoremstyle{definitions}
\newtheorem{definition}[theorem]{Definition}

\newtheorem{remark}[theorem]{Remark}

\theoremstyle{notations}

\theoremstyle{remarks}

\journal{ }

\begin{document}

\begin{frontmatter}
\title{The Capacity of Wedge Sum of Spheres of Different Dimensions}
\author[]{Mojtaba~Mohareri}
\ead{m.mohareri@stu.um.ac.ir}
\author[]{Behrooz~Mashayekhy\corref{cor1}}
\ead{bmashf@um.ac.ir}
\author[]{Hanieh~Mirebrahimi}
\ead{h$_{-}$mirebrahimi@um.ac.ir}
\address{Department of Pure Mathematics, Center of Excellence in Analysis on Algebraic Structures, Ferdowsi University of
Mashhad,\\
P.O.Box 1159-91775, Mashhad, Iran.}
\cortext[cor1]{Corresponding author}
\begin{abstract}
 K. Borsuk in 1979, in the Topological Conference in Moscow, introduced the concept of the capacity of a compactum and raised some interesting questions about it.  In this paper, during computing the capacity of wedge sum of finitely many spheres of different dimensions and the complex projective plane, we give a negative answer to a question of Borsuk whether the capacity of a compactum determined by its homology properties.
\end{abstract}


\end{frontmatter}
\section{Introduction and Motivation}
 K. Borsuk in \cite{So}, introduced the concept of the capacity of a compactum (compact metric space) as follows:
 the capacity $C(A)$ of a compactum $A$ is the cardinality of the set of all shapes of compacta $X$ for which $\mathcal{S}
h(X) \leqslant \mathcal{S}h(A)$ (for more details, see \cite{Mar}).

For polyhedra,  the notions shape and shape domination in the above definition can be replaced by the notions homotopy type and homotopy domination, respectively. Indeed, by some known
results in shape theory one can conclude that for any polyhedron $P$, there is a one to one functorial correspondence between the shapes of compacta shape dominated by $P$ and the homotopy types of CW-complexes (not necessarily finite) homotopy dominated by $P$ (for both pointed and unpointed polyhedra) \cite{18}.

S. Mather in \cite{17} proved that every polyhedron dominates only countably many  different homotopy types (hence shapes). Note that the capacity of a topological space is a homotopy invariant, i.e., for topological spaces $X$ and $Y$ with the same homotopy type, $C(X)=C(Y)$. Hence it seems interesting to find topological spaces with finite capacity and compute the capacity of some of their well-known spaces.  Borsuk in \cite{So} asked a question: `` Is it true that the capacity of every finite polyhedron is finite? ''.  D. Kolodziejczyk in \cite{16} gave a negative answer to this question. Also, she investigated some conditions for polyhedra to have finite capacity (\cite{13, 18, 14, 15}). For instance, a polyhedron $Q$ with finite fundamental group $\pi_1 (Q)$ and a polyhedron $P$ with abelian fundamental group $\pi_1 (P)$ and finitely generated homology groups $H_i (\tilde{P})$, for $i\geq 2$ where $\tilde{P}$ is the universal cover of $P$, have finite capacities.

Borsuk in \cite{So} mentioned that the capacity of $\bigvee_k \mathbb{S}^1$ and $\mathbb{S}^n$ equals to $k+1$ and 2, respectively.  The authors in \cite{Moh} computed the capacity of Moore spaces $M(A,n)$ and Eilenberg-MacLane spaces $K(G,n)$. In fact, we showed that the capacities of a Moore space $M(A,n)$ and an Eilenberg-MacLane space $K(G,n)$ equal to the number of direct summands of $A$ and semidirect factors of  $G$, respectively, up to isomorphism. Also, we computed the capacity of  wedge sum of finitely many Moore spaces of different degrees and the capacity of the product of finitely many Eilenberg-MacLane spaces of different homotopy types. In particular, we showed that the capacity of $\bigvee_{n\in I} (\vee_{i_n} \mathbb{S}^n)$ equals to $\prod_{n\in I}(i_n +1)$ where $\vee_{i_n} \mathbb{S}^n$ denotes the wedge sum of $i_n$ copies of $\mathbb{S}^n$, $I$ is a finite subset of $\mathbb{N}\setminus \{ 1\}$ and $i_n \in \mathbb{N}$.

In Section 2, we will state some basic facts that we need for the rest of the paper. In Section 3, we will compute the capacity of wedge sum of finitely many spheres of different dimensions (containing spheres of dimension 1). As a consequence, we show that  two dimensional CW-complexes with free fundamental groups  have finite capacities.

Borsuk in \cite{So} stated some questions concerning properties of the capacity of
compacta. One of them is as follow:
\begin{center}
  ``Is the capacity $C(A)$ determined by the homology properties of $A$?''
\end{center}
In Section 4, we give a negative answer to this question, in general.  However, the authors in \cite{Moh}  showed that the capacity of Moore spaces determined by their homology properties.

\section{Preliminaries}
In this paper, every CW-complex is assumed to be finite and connected and every map between two CW-complexes is assumed to be cellular. We assume that the reader is familiar with the basic notions and facts of homotopy theory. We need the following results and definitions for the rest of the paper.
\begin{definition}\cite{Wall}.
Let $\phi :K\longrightarrow X$ be a map between CW-complexes with mapping cylinder $M=X\bigcup_{\phi} (K\times I)$. Denote $\pi_n (M,K\times \{ 1\} )$  by $\pi_n (\phi )$. The map $\phi$ is called $n$-connected if $K$ and $X$ are connected and $\pi_i (\phi )=0$ for $1\leq i\leq n$.
\end{definition}
\begin{theorem}\cite[Theorem 4.37]{3}. \label{Hat}
If $(X,A)$ is an $(n-1)$-connected pair of path-connected spaces
with $n \geq 2$ and $A\neq \emptyset$, then the Hurewicz homomorphism $h_n :\pi_n (X,A)\longrightarrow H_n (X,A)$ is an isomorphism and $H_i (X,A)=0$ for $i<n$.
\end{theorem}
Recall that a pair $(X,A)$ of topological spaces is called $n$-connected $(n\geq 0)$ if $\pi_i (X,A)=0$ for $0\leq i\leq n$.
\begin{theorem}\cite[Proposition 4.1]{3}. \label{Hat1}
A covering projection $p:(\tilde{X},\tilde{x}_0 )\longrightarrow (X,x_0 )$ induces isomorphisms
$p_* :\pi_n (\tilde{X},\tilde{x}_0 )\longrightarrow \pi_n (X,x_0 )$ for all $n \geq 2$.
\end{theorem}
\begin{definition}\cite{Wall}.
Let $X$ be a CW-complex. Then conditions $\mathcal{F}_i$ and $\mathcal{D}_i$ on $X$ are defined inductively as follows:

$\mathcal{F}_1$: the group $\pi_1 (X)$ is finitely generated.

$\mathcal{F}_2$: the group $\pi_1 (X)$ is finitely presented, and for any  2-dimensional finite CW-complex $K$ and any map $\phi :K\longrightarrow X$ inducing an isomorphism of fundamental groups, $\pi_2 (\phi )$is a finitely generated module over $\mathbb{Z}\pi_1 (X)$.

$\mathcal{F}_n$: the condition $\mathcal{F}_{n-1}$ holds, and for any $(n-1)$-dimensional finite CW-complex $K$ and any $(n-1)$-connected map $\phi :K\longrightarrow X$, $\pi_n (\phi )$ is a finitely generated $\mathbb{Z}\pi_1 (X)$-module.

$\mathcal{D}_n$: $H_i(\tilde{X})=0$ for $i>n$, and $H^{n+1}(X;\mathcal{B})=0$ for all coefficient bundles $\mathcal{B}$ (for more detalis, see \cite{STE}).
\end{definition}
\begin{theorem}\cite[Theorem F]{Wall}.\label{700}
A CW-complex $X$ is dominated by a finite CW-complex of dimension $n\geq 1$ if and only if $X$ satisfies $\mathcal{D}_n$ and $\mathcal{F}_n$.
\end{theorem}
\begin{definition}\cite{Rob}.
A group $G$ is called Hopfian if every epimorphism $f :G\longrightarrow G$ is an automorphism
(equivalently, $N = 1$ is the only normal subgroup of $G$ for which $G/N\cong G$).
\end{definition}
 \begin{definition}\citep{4}.
A Moore space of degree $n$ $(n\geq 2)$ is a simply connected $CW$-complex $X$ with a single non-vanishing homology group of degree $n$, that is $\tilde{H}_{i}(X,\mathbb{Z})=0$ for $i\neq n$. A Moore space is denoted by $M(A,n)$ where $A\cong \tilde{H}_{n} (X,\mathbb{Z})$.
\end{definition}
Note that for $n=1$, the Moore space $M(A,1)$ can not be defined because of some problems in existence and uniqueness of the space (for more details see \cite{3}).
\begin{theorem}\label{-2}\cite{4}.
The homotopy type of a  Moore space $M(A,n)$ is uniquely determined by $A$ and $n$ ($n > 1$).
\end{theorem}
\begin{theorem}\label{3030}\cite{3}.
If a map $f :X\longrightarrow Y$ between connected CW-complexes induces isomorphisms
$f_* :\pi_n (X)\longrightarrow \pi_n (Y)$  for all $n\geq 1$, then $f$ is a homotopy equivalence.
\end{theorem}
\begin{theorem}\label{3031}\cite{3}.
A map $f :X\longrightarrow Y$ between simply-connected CW-complexes is a homotopy
equivalence if $f_* :H_n (X) \longrightarrow H_n (Y) $  is an isomorphism for all $n\geq 2$.
\end{theorem}
\begin{theorem}\label{-1}\cite{4}.
1) A connected CW-complex $X$ is contractible if and only if all its homotopy groups $\pi_n (X)$ ($n\geq 1$) are trivial.

2) A simply connected CW-complex $X$ is contractible if and only if all its homology groups $H_n (X)$ ($n\geq 2$) are trivial.
\end{theorem}
\begin{theorem}\cite{1}
\[
H_p (\mathbb{CP}^n )=
\begin{cases}
\mathbb{Z} &  p=0, 2, 4, \cdots , 2n \\
0 & otherwise.
\end{cases}
\]
\end{theorem}
\begin{theorem}\cite[Theorem 3.15]{Cohen}. \label{Cohen}
Let $(X,E,\Phi )$ be a CW-complex and $(\tilde{X},p)$ be its universal covering. For each cell $e_\alpha$ of $X$, let an specific characteristic map $\phi_\alpha :I^n \longrightarrow 	X$ ($n=n(\alpha )$) and an specific lift $\tilde{\phi}_\alpha :I^n \longrightarrow \tilde{X}$ of $\phi_\alpha$ be chosen.  Then $\{ \langle \tilde{\phi}_\alpha \rangle \; | \; e_\alpha \in X \}$ is a basis for $C(\tilde{X})$ as a $\mathbb{Z}\pi_1 (X)$-module.
\end{theorem}
\section{The Capacity of Wedge Sum of  Spheres}
Borsuk in \cite{16} mentioned that the capacity of $\bigvee_k \mathbb{S}^1$ equals to $k+1$ and $C (\mathbb{S}^n )=2$ for $n\geq 1$.  The authors in \cite{Moh} showed that the capacity of a Moore space $M(A,n)$ equals to the number of direct summands of $A$, up to isomorphism. Also, we computed the capacity of the wedge sum of finitely many Moore spaces of different degrees. In particular, we showed that the capacity of $\bigvee_{n\in I} (\vee_{i_n} \mathbb{S}^n)$ equals to $\prod_{n\in I}(i_n +1)$ where $\vee_{i_n} \mathbb{S}^n$ denotes the wedge sum of $i_n$ copies of $\mathbb{S}^n$, $I$ is a finite subset of $\mathbb{N}\setminus \{ 1\}$ and $i_n \in \mathbb{N}$.  As a special case, the capacity of $\mathbb{S}^m \vee \mathbb{S}^n$ $(m,n\geq 2, m\neq n)$ equals to 4.

In this section, we compute the capacity of wedge sum of finitely many spheres of different dimensions (containing spheres of dimension 1). As a consequence, we prove that  two dimensional CW-complexes with free fundamental groups have finite capacities.
\begin{remark}\cite{Wall}. \label{REM}
Suppose that the fundamental group  $\pi_1 (X)$ of a CW-complex $X$ is a free group of rank $r$. One can find a map $\phi :K\longrightarrow X$ inducing an isomorphism of fundamental groups, where $K$ is wedge sum of $r$ copies of $\mathbb{S}^1$.
\end{remark}
\begin{lemma}\label{tamim}
Suppose that the fundamental group $\pi_1 (X)$ of a CW-complex $X$ is a free group of rank $r$ and  $\phi$ is as in Remark \ref{REM}. Also, assume that $\pi_i (\phi )$ is a free  $\mathbb{Z}\pi_1 (X)$-module of rank $r_i$ for $i\geq 2$ and $H_k (\tilde{X})=0$ for $k>n$. Then $X$ has the homotopy type of wedge sum $\bigvee_{r}\mathbb{S}^1 \vee \big( \bigvee_{i=2}^{n} (\vee_{r_i} \mathbb{S}^i) \big)$ where $\vee_{r_i} \mathbb{S}^i$ denotes the wedge sum of $r_i$  copies of $\mathbb{S}^i$.
\end{lemma}
\begin{proof}
It is easy to see that $\phi$ is 1-connected ($\phi$ induces an isomorphism of fundamental groups). Using the construction introduced in \cite[p. 59]{Wall}  and by induction on $i$, since by the hypothesis $\pi_i (\phi )$ is a free $\mathbb{Z}\pi_1 (X)$-module, we can attach finitely many $i$-cells to $K$, obtaining $K^{(i)}$ say, necessarily with trivial attaching maps to make $\phi$ $i$-connected. By Theorem \ref{3030}, $K^{(n)}$ is homotopy equivalent to $X$.  By hypothesis $C_k (\tilde{X})=H_k (\tilde{X})=0$ for $k>n$ ($B_k (\tilde{X})=0$ because of the shape of $\tilde{X}$) and hence by Theorem \ref{Cohen}, $C_k (X)=0$ for $k>n$. Thus  $K^{(n)}$ is the wedge sum $\bigvee_{r}\mathbb{S}^1 \vee \big( \bigvee_{i=2}^{n} (\vee_{r_i} \mathbb{S}^i) \big)$.
\end{proof}
Note that Lemma \ref{tamim} is a generalization of the following result:
\begin{theorem}\cite[Proposition 3.3]{Wall}. \label{701}
If a CW-complex $X$ satisfies $\mathcal{D}_2$ and $\mathcal{F}_2$ and $\pi_1 (X)$ is a  free group, then $X$ has the homotopy type of a finite bouquet of 1-spheres and 2-shperes.
\end{theorem}
\begin{theorem}\label{ASL}
The capacity of $\bigvee_{n\in I} (\vee_{i_n} \mathbb{S}^n)$ equals to $\prod_{n\in I}(i_n +1)$ where $\vee_{i_n} \mathbb{S}^n$ denotes the wedge sum of $i_n$ copies of $\mathbb{S}^n$, $I$ is a finite subset of $\mathbb{N}$ and $i_n \in \mathbb{N}$.
\end{theorem}
\begin{proof}
If  $I\subset \mathbb{N}\setminus \{1 \}$, the result has been proved by the authors in \cite{Moh}. Hence we can suppose that $X=\bigvee_{r}\mathbb{S}^1 \vee \big( \bigvee_{n\in I\setminus \{ 1\}} (\vee_{i_n} \mathbb{S}^n) \big)$. Assume that $A$ is homotopy dominated by $X$ and $\pi_1 (A)$ is of rank $s$ where $0\leq s\leq r$. By Remark \ref{REM}, one can find a map $\phi :K \longrightarrow A$ induces an isomorphism of fundamental groups, where $K$ is wedge sum of $s$ copies of $\mathbb{S}^1$ . By Theorems \ref{Hat1} and \ref{Hat}, we have $\pi_n (\phi ) \cong \pi_n (\tilde{\phi}) \cong H_n (\tilde{A},\tilde{K})\cong  \frac{C_n (\tilde{A})}{B_n (\tilde{A})}$. Note that $B_n (\tilde{A})\leqslant B_n (\tilde{X})$ and $B_n (\tilde{X})=0$ since $\tilde{X}$ is a tree with a $ \bigvee_{n\in I\setminus \{ 1\}} (\vee_{i_n} \mathbb{S}^n)$ on each of its vertex, hence $B_n (\tilde{A})=0$. This shows that $\pi_n (\phi )\cong C_n (\tilde{A})$ as $\mathbb{Z}\pi_1 (A)$-module.  By Theorem \ref{Cohen}, the rank of $C_n (\tilde{A})$ as $\mathbb{Z}\pi_1 (A)$-module equals to the rank of $C_n (A)$ as $\mathbb{Z}$-module (which is a submodule of $C_n (X)$). Therefore $\pi_n (\phi )$ is a free $\mathbb{Z}\pi_1 (A)$-module of rank $j_n$ ($0\leq j_n \leq i_n$). Hence by Lemma \ref{tamim}, $A$ is homotopy equivalent to $\bigvee_{s}\mathbb{S}^1 \vee \big( \bigvee_{n\in I\setminus \{ 1\}} (\vee_{j_n} \mathbb{S}^n) \big)$.
\end{proof}
\begin{corollary}
The capacity of $\mathbb{S}^1 \vee \mathbb{S}^n$ equals to 4,  for every $n\geq 2$.
\end{corollary}
\begin{remark}
 Kolodziejczyk in \cite{18} asked the following question:
\begin{center}
 ``Does every polyhedron P with the abelian fundamental group $\pi_1 (P)$ dominate only finitely many different homotopy types?''
\end{center}
She proved that two large classes of polyhedra, polyhedra $Q$ with finite fundamental groups $\pi_1 (Q)$, and polyhedra $P$ with abelian fundamental groups and finitely generated homology groups $H_i (\tilde{P})$  $(i\geq 2)$,  have finite capacities, where $\tilde{P}$ is the universal covering of $P$ (see \cite{14},\cite{18}). Note that the wedge sum $\mathbb{S}^1 \vee \mathbb{S}^2$ is a simple example of a polyhedron $P$ with infinite abelian fundamental group $\pi_1 (P )$ and infinitely generated homology group $H_2 (\tilde{P}; \mathbb{Z})$ which  is not wedge sum of Moore spaces. So its capacity can not determined  by neither the results of Kolodziejczyk nor the results of \cite{Moh}.
\end{remark}
\begin{corollary}
Every 2-dimensional CW-complex $X$ with free fundamental group $\pi_1 (X)$ has finite capacity. Moreover, if the ranks of $\pi_1 (X)$ and $H_2 (X)$ are $r$ and $s$, respectively, then the capacity of $X$ equals to $(r+1)\times (s+1)$.
\end{corollary}
\begin{proof}
It can be concluded from Theorems \ref{700}, \ref{701} and \ref{ASL}.
\end{proof}
\section{A Negative Answer to a Question of Borsuk}
Borsuk in \cite{So} stated some questions concerning properties of the capacity of
compacta. One of them is as follow:
\begin{center}
  Is the capacity $C(A)$ determined by the homology properties of $A$?
\end{center}
In this section, we give a negative answer to this question, in general.  However, the authors in \cite{Moh}  showed that the capacity of Moore spaces determined by their homology properties. In fact, they proved that the capacity of a Moore space $M(A,n)$ equals to the number of direct summands of $A$, up to isomorphism.
\begin{lemma}\cite[Example 2.43]{3}.\label{0}
Let $X$ be the Eilenberg-MacLane space $K(\mathbb{Z}_m ,1)$. Then $H_n (X)$ is $\mathbb{Z}_m$ for odd $n$ and $0$ for even $n>0$.
\end{lemma}
\begin{lemma}\label{L2}
The capacity of projective plane $\mathbb{CP}^2$ equals to 2.
\end{lemma}
\begin{proof}
Suppose that $A$ is homotopy dominated by $\mathbb{CP}^2$. Then $A$ is  simply connected and $H_n (A)$ is isomorphic to a direct summand of  $H_n (\mathbb{CP}^2 )$. Since $H_n (\mathbb{CP}^2 )\cong \mathbb{Z}$ for $n=0,2,4$ and $H_n (\mathbb{CP}^2)=0$ otherwise, we have the following cases:
\begin{enumerate}[a)]
\item
If $H_2 (A),H_4 (A)=0$, then by part (2) of Theorem \ref{-1}, $A$ is homotopy equivalent to a one-point space.
\item
If $H_2 (A)\cong \mathbb{Z}$ and $H_4 (A)=0$, then $A$ is the Moore space $M(\mathbb{Z},2)$ and so, by Theorem \ref{-2}, it is homotopy equivalent to $\mathbb{S}^2$. But $\mathbb{S}^2$ can not be homotopy dominated by $\mathbb{CP}^2$. Since if $\mathbb{S}^2 \leqslant \mathbb{CP}^2$, then $\pi_3 (\mathbb{S}^2 )\cong \mathbb{Z}$ is isomorphic to a subgroup of $\pi_3 (\mathbb{CP}^2 )\cong \pi_3 (\mathbb{S}^5)=0$ which is a contradiction (note that by \cite[Exercise 11.47]{1}, $\pi_q (\mathbb{CP}^n )\cong \pi_q (\mathbb{S}^{2n+1} )$ for all $q\geq 3$).
\item
If $H_2 (A)=0$ and $H_4 (A)\cong \mathbb{Z}$, then $A$ is the Moore space $M(\mathbb{Z},4)$ and so, by Theorem \ref{-2}, it is homotopy equivalent to $\mathbb{S}^4$. Similar to the case (b), $\mathbb{S}^4$ can not be homotopy dominated by $\mathbb{CP}^2$ since $\pi_4 (\mathbb{S}^4 )\cong \mathbb{Z}$ is not isomorphic to a subgroup of $\pi_4 (\mathbb{CP}^2 )\cong \pi_4 (\mathbb{S}^5)=0$.
\item
 Let $H_2 (A)\cong \mathbb{Z}$ and $H_4 (A)\cong \mathbb{Z}$. Suppose $g:\mathbb{CP}^2 \longrightarrow A$ is the domination map. It is easy to see that $g_* : H_n (\mathbb{CP}^2 )\longrightarrow H_n (A)$ is an  epimorphism between two isomorphic Hopfian groups for all $n\geq 2$ which implies that $g_*$ is isomorphism for all $n\geq 2$. Therefore by Theorem \ref{3031}, $g$ is a homotopy equivalence and hence, $A$ and $\mathbb{CP}^2$ have the same homotopy type.
\end{enumerate}
Thus $A$ can only be homotopy equivalent to either one-point space or $\mathbb{CP}^2$ and so, $C(\mathbb{CP}^2 )=2$.
\end{proof}

Now, we are in a position to give a negative answer to the question of Borsuk. 
\begin{theorem}
The capacity of a compactum is not determined by its homology properties, in general.
\end{theorem}
\begin{proof}
Consider  $X = \mathbb{S}^2 \vee \mathbb{S}^4$ and $Y= \mathbb{CP}^2$. It is easy to see that $H_n (X) \cong H_n(Y)$ for  all $n\geq 0$ ($ H_n (X)\cong H_n (Y)\cong \mathbb{Z}$ for $n = 0, 2, 4$ and $H_n (X) =H_n (Y)= 0$ otherwise). But by Theorem \ref{ASL}, $C(X)=4$ and by Lemma \ref{L2}, $C(Y)=2$.
\end{proof}
\begin{remark}
Note that the capacity of a compactum is not also determined by its homotopy groups. Consider $X=\mathbb{S}^2$ and $Y=\mathbb{S}^3 \times K(\mathbb{Z},2)$. For these simply connected spaces, we have $\pi_i (X)\cong \pi_i (Y)$ for all $i\geq 2$, however $C (X)=2$ and  $C(Y)\geq 4$ since $\{ *\}$, $\mathbb{S}^3$, $K(\mathbb{Z}, 2)$ and $Y$ are homotopy dominated by $Y$ which are not homotopy equivalent to each other.
\end{remark}


\begin{thebibliography}{9999}
\bibitem{4}  H.J. Baues, Homotopy Type and Homology, Oxford University Press Inc., New York, 1996.
\bibitem{So}  K. Borsuk, Some problems in the theory of shape of compacta, Russian Math. Surveys 34 (6) (1979) 24-26.
\bibitem{Cohen} M. Cohen, A Course in Simple Homotopy Theory, Graduate Texts in Mathematics,
vol. 10, Springer, Berlin, 1970.
\bibitem{3}  A. Hatcher, Algebraic Topology, Cambridge University Press, 2002.
\bibitem{13}  D. Kolodziejczyk, Homotopy dominations within polyhedra, Fund. Math. 178 (2003) 189-202.
\bibitem{18} D. Kolodziejczyk, Polyhedra dominating finitely many different homotopy types, Topology Appl. 146-147
(2005), 583-592.
\bibitem{14}  D. Kolodziejczyk, Polyhedra with finite fundamental group dominate only finitely many different homotopy types, Fund. Math. 180 (2003) 1-9.
\bibitem{15}  D. Kolodziejczyk, Simply-connected polyhedra dominate only finitely many different shapes, Topology Appl. 112 (3) (2001) 289-295.
\bibitem{16}   D. Kolodziejczyk, There exists a polyhedron dominating infinitely many different homotopy types, Fund. Math. 151 (1996) 39-46.
\bibitem{Mar}  S. Mardesic, J. Segal, Shape Theory. The Inverse System Approach, North-Holland Mathematical Library, vol. 26, North-Holland, Amsterdam, 1982.
\bibitem{17} M. Mather, Counting homotopy types of manifolds, Topology 4 (1965) 93-94.
\bibitem{Moh} M. Mohareri, B. Mashayekhy and H. Mirebrahimi, On The Capacity of Eilenberg-MacLane and
Moore Spaces, arXiv:1607.07054.
\bibitem{Rob}  D.J.S. Robinson, A Course in the Theory of Groups, Springer, Berlin, 1982.
\bibitem{1}  J.J. Rotman, An Introduction to Algebraic Topology, G.T.M. 119, Springer-Verlag,
New York, 1988.
\bibitem{STE} N. E. Steenrod, The Topology of Fiber Bundles, Princeton, 1951.
\bibitem{Wall} C.T.C. Wall, Finiteness conditions for CW-complexes, Ann. of Math. 81 (1965) 56–69.
\end{thebibliography}
\end{document}